\documentclass{amsart} 
\usepackage{amssymb}
\usepackage{amsmath}
\usepackage{amsfonts}

\begin{document}
\newtheorem{theo}{Theorem}[section]
\newtheorem{atheo}{Theorem*}
\newtheorem{prop}[theo]{Proposition}
\newtheorem{aprop}[atheo]{Proposition*}
\newtheorem{lemma}[theo]{Lemma}
\newtheorem{alemma}[atheo]{Lemma*}
\newtheorem{exam}[theo]{Example}
\newtheorem{coro}[theo]{Corollary}
\theoremstyle{definition}
\newtheorem{defi}[theo]{Definition}
\newtheorem{rem}[theo]{Remark}


\newcommand{\Bb}{{\bf B}}
\newcommand{\Cb}{{\mathbb C}}
\newcommand{\Nb}{{\mathbb N}}
\newcommand{\Qb}{{\mathbb Q}}
\newcommand{\Rb}{{\mathbb R}}
\newcommand{\Zb}{{\mathbb Z}}
\newcommand{\Ac}{{\mathcal A}}
\newcommand{\Bc}{{\mathcal B}}
\newcommand{\Cc}{{\mathcal C}}
\newcommand{\Dc}{{\mathcal D}}
\newcommand{\Fc}{{\mathcal F}}
\newcommand{\Ic}{{\mathcal I}}
\newcommand{\Jc}{{\mathcal J}}
\newcommand{\Kc}{{\mathcal K}}
\newcommand{\Lc}{{\mathcal L}}
\newcommand{\Oc}{{\mathcal O}}
\newcommand{\Pc}{{\mathcal P}}
\newcommand{\Sc}{{\mathcal S}}
\newcommand{\Tc}{{\mathcal T}}
\newcommand{\Uc}{{\mathcal U}}
\newcommand{\Vc}{{\mathcal V}}

\author{Nik Weaver}

\title [The quantum Tur\'{a}n problem]
       {The ``quantum'' Tur\'{a}n problem for operator systems}

\address {Department of Mathematics\\
Washington University\\
Saint Louis, MO 63130}

\email {nweaver@math.wustl.edu}

\date{\em Feb.\ 20, 2018}


\begin{abstract}
Let $\mathcal{V}$ be a linear subspace of $M_n(\mathbb{C})$ which contains
the identity matrix and is stable under Hermitian transpose.
A ``quantum $k$-clique'' for $\mathcal{V}$ is a rank $k$ orthogonal projection
$P \in M_n(\mathbb{C})$ for which ${\rm dim}(P\mathcal{V}P) = k^2$, and a
``quantum $k$-anticlique'' is a rank $k$ orthogonal projection for which
${\rm dim}(P\mathcal{V}P) = 1$. We give upper and lower bounds both for
the largest dimension of $\mathcal{V}$ which would ensure the existence of
a quantum $k$-anticlique, and for the smallest dimension of $\mathcal{V}$
which would ensure the existence of a quantum $k$-clique.
\end{abstract}

\maketitle

\section{Background}

In finite dimensions, an {\it operator system} is a linear subspace
$\mathcal{V}$ of $M_n(\mathbb{C})$ with the properties
\begin{itemize}
\item $I_n \in \mathcal{V}$

\item $A \in \mathcal{V} \Rightarrow A^* \in \mathcal{V}$
\end{itemize}
where $I_n$ is the $n\times n$ identity matrix and $A^*$ is the Hermitian
transpose of $A$.

A natural class of examples arises from graphs with vertex set
$\{1, \ldots, n\}$. Given such a graph $G$, we can define an operator system
$$\mathcal{V}_G = {\rm span}\{E_{ij}:
i = j\mbox{ or $i$ is adjacent to $j$}\}$$
where $E_{ij}$ is the $n\times n$ matrix with a $1$ in the $(i,j)$ entry
and $0$'s elsewhere. Note that the symmetry of the edge set of $G$ is reflected
in the stability of $\mathcal{V}_G$ under Hermitian transpose. (These are
precisely the operator systems which are bimodules over the diagonal subalgebra
of $M_n(\mathbb{C})$.)

Operator systems have been studied by C*-algebraists for decades, but only
recently have they begun to be thought of as being, in some way, a matrix
or ``quantum'' analog of graphs. More generally, we can regard the notion
of a linear subspace of $M_n(\mathbb{C})$ as a linearization of the notion of a subset of
$\{1, \ldots, n\}^2$, i.e., a relation on the set $\{1, \ldots, n\}$. The
two conditions which define operator systems are then matrix versions of
reflexivity and symmetry, so that an operator system becomes
a matrix version of a reflexive, symmetric relation on a set ---
which is effectively the same as a graph on that set.\footnote{There is an obvious 1-1 correspondence between graphs on a vertex set $V$ and reflexive, symmetric relations on $V$. This correspondence is more natural if we adopt the convention that graphs must have a loop at each vertex; in the error correction setting discussed below, where an edge between two vertices expresses that they are ``sufficiently close'', this is in fact a good convention.}
This point of view was developed in \cite{W,W2}.

The term ``quantum'' is supported by the fact that operator systems
appear in the theory of quantum error correction, playing a role exactly
analogous to the role played by ordinary graphs in classical error correction
\cite{DSW}. In the classical case we have a {\it confusability graph} which
tells us when two transmitted signals could be received as the same signal,
and in the quantum case we have a {\it confusability operator system} which
tells us when two transmitted states could be received as the same state. The
two settings even have a natural common generalization; see \cite{W2}.

The first paper to demonstrate that there could be a ``quantum graph theory''
for operator systems was \cite{DSW}, where, driven by the needs of quantum
error correction, a ``quantum Lov\'{a}sz number'' was defined for an arbitrary
operator system, in analogy to the classical Lov\'{a}sz number of a graph.
The error correction perspective on quantum graphs was developed further in
\cite{S}.

The present paper is a sequel to \cite{W3}, where an operator system
version of Ramsey's theorem was proven. This result involves quantum versions
of graph-theoretic cliques and anticliques. The theory of error correction
tells us what a quantum anticlique should be, because in classical error
correction a
``code'' is realized as an anticlique in the confusability graph, whereas
in quantum error correction a ``code'' is realized as an orthogonal projection
$P \in M_n(\mathbb{C})$ satisfying $PAP = \lambda P$ for all $A$ belonging to the
confusability operator system $\mathcal{V}$. Equivalently, this condition
can be stated as
${\rm dim}(P\mathcal{V}P) = 1$ where $P\mathcal{V}P = \{PAP: A \in \mathcal{V}\}$.

Observe that if $P \in M_n(\mathbb{C})$ is any orthogonal projection (i.e.,
$P = P^2 = P^*$) and $\mathcal{V} \subseteq M_n(\mathbb{C})$ is any operator system,
then $P\mathcal{V}P$ is effectively a set of linear
transformations from ${\rm ran}(P)$ to itself, and the condition that
$P$ should be a code is that this set should be minimal, consisting only
of the scalar multiples of the identity operator on ${\rm ran}(P)$. If
these are the anticliques of $\mathcal{V}$, then it is natural to take
the cliques of $\mathcal{V}$ to be the orthogonal projections $P$ for
which $P\mathcal{V}P$ is maximal, i.e., it consists of all linear operators
from ${\rm ran}(P)$ to itself. This can also be expressed by saying that
${\rm dim}(P\mathcal{V}P) = k^2$. We therefore make the following definition.

\begin{defi}
Let $\mathcal{V} \subseteq M_n(\mathbb{C})$ be an operator system. A rank $k$ orthogonal
projection $P \in M_n(\mathbb{C})$ is a {\it quantum $k$-anticlique} for $\mathcal{V}$
if ${\rm dim}(P\mathcal{V}P) = 1$, and a {\it quantum $k$-clique} for
$\mathcal{V}$ if ${\rm dim}(P\mathcal{V}P) = k^2$.
\end{defi}

In general, if we identify $PM_n(\mathbb{C})P$ with $M_k(\mathbb{C})$, where $k = {\rm rank}(P)$,
then $P\mathcal{V}P$ becomes an operator system in $M_k(\mathbb{C})$. This is the
{\it induced operator system} which is analogous to a subgraph induced on
a subset of the vertex set of a graph. (Some intuition for this analogy
is given in \cite{W2}, again with the natural common generalization
mentioned earler.) Thus $P$ is a quantum clique if the
induced operator system is a full matrix algebra and it is a quantum
anticlique if the induced operator system is trivial.

The classical theorem of Ramsey states that for any $k$ there exists $n$ such
that every graph with $n$ vertices has either a $k$-clique or a
$k$-anticlique. The quantum Ramsey theorem proven in
\cite{W3} states that for any $k$ there exists $n$ such that every operator
system in $M_n(\mathbb{C})$ has either a quantum $k$-clique or a quantum $k$-anticlique.
The most surprising aspect of this result is that in the quantum setting
$n$ grows polynomially in $k$, not exponentially as in the classical case.
(The specific value given in \cite{W3} is $n = 8k^{11}$, but this is surely
not optimal. An easy lower bound is $n = (k-1)(k^2 - 1) = k^3 - k^2 - k + 1$,
obtained by taking
$r = k^2 - 1$ in the construction described in Proposition \ref{projsum}
below.) A quantum Ramsey theorem for infinite-dimensional operator systems
was proven in \cite{KKS}.

Michael Jury suggested to me the problem of finding a version of
Tur\'{a}n's theorem for operator systems. The classical theorem of Tur\'{a}n
gives the maximum number of edges a graph with $n$ vertices can have
without having any $(k+1)$-cliques; by taking edge complements, we see that
${n\choose 2}$ minus this number is the minimum number of edges a graph with
$n$ vertices can have without having any $(k+1)$-anticliques. The analogous
questions for operator systems are: what is the maximum dimension
$T^{\uparrow}(n,k)$ of an operator system in
$M_n(\mathbb{C})$ having no quantum $(k+1)$-cliques,
and what is the minimum dimension $T^{\downarrow}(n,k)$ of an operator
system in $M_n(\mathbb{C})$ having
no quantum $(k+1)$-anticliques. These two questions constitute a ``quantum
Tur\'{a}n problem''. The goal of this paper is not to give
exact answers to them, but merely to provide upper and lower
bounds for both values. Specifically, we prove
$$\sqrt{\frac{n}{k}} < T^{\downarrow}(n,k)
\leq \left\lceil\frac{n}{k}\right\rceil\quad\mbox{and}\quad
2(k-1)n - (k-1)^2 + 3 \leq T^{\uparrow}(n,k) < 16(k+1)^8n.$$

Because, unlike the classical case, there is no natural symmetry between
quantum cliques and quantum anticliques, we are really dealing with two
distinct questions. Broadly speaking, it is {\it easy} to find quantum
cliques and {\it hard} to find quantum anticliques. This is dramatically
illustrated by the fact that our upper bound on the maximum dimension of
an operator system having no quantum $(k+1)$-cliques is linear in $n$. As
there are $n^2$ available dimensions in $M_n(\mathbb{C})$,
this means that when $n$ is large compared to $k$ one needs only a
comparatively small number of dimensions to guarantee that quantum
$(k+1)$-cliques exist. In contrast, the upper bound on the lower quantum
Tur\'{a}n number is $\lceil{\frac{n}{k}}\rceil$, meaning that
${\rm dim}(\mathcal{V})$ has to be even smaller than this to ensure that
quantum $(k+1)$-anticliques exist.

I gratefully acknowledge Robert Bryant's essential contributions to the
material in the last part of Section 2.

\section{Lower quantum Tur\'{a}n numbers}

We define the {\it lower quantum Tur\'{a}n number} $T^{\downarrow}(n,k)$
to be the smallest number $d$ such that some operator system in $M_n(\mathbb{C})$
whose dimension is $d$ has no quantum $(k+1)$-anticliques.

Every rank $1$ projection is always both a quantum $1$-anticlique and a
quantum $1$-clique for any
operator system, so let us assume throughout that $k \geq 1$.

Classically, a graph on $n$ vertices which lacks $(k+1)$-anticliques, and has
the minimum number of edges for doing so, looks like a disjoint union of
$k$ many cliques of equal or nearly equal size. So a natural guess for
an operator system in $M_n(\mathbb{C})$ which lacks quantum $(k+1)$-anticliques and has
the smallest possible dimension is a direct sum of $k$ many matrix
algebras of equal or nearly equal size, $\mathcal{V} = M_{n_1}(\mathbb{C}) \oplus
\cdots \oplus M_{n_{k-1}}(\mathbb{C})$. This operator system indeed has no quantum
$(k+1)$-anticliques; in fact, it has no quantum $2$-anticliques because it
contains the diagonal operator system $D_n$, which itself has no quantum
$2$-anticliques \cite[Proposition 2.1]{W3}. But this shows that this
$\mathcal{V}$ is far from being minimal: its dimension is
approximately $\frac{n^2}{k}$, whereas the dimension of $D_n$ is $n$.
Quantum $(k+1)$-anticliques for $k > 1$ can be blocked using even fewer
dimensions.

\begin{prop}\label{projsum}
Let $P_1, \ldots, P_r$ be orthogonal projections in $M_n(\mathbb{C})$, each of
rank at most $k$, satisfying $P_1 + \cdots + P_r = I_n$. Then the operator
system $\mathcal{V} = {\rm span}(P_1, \ldots, P_r)$ has no quantum
$(k+1)$-anticliques.
\end{prop}

\begin{proof}
Let $P$ be a rank $k+1$ orthogonal projection in $M_n(\mathbb{C})$ and assume
$P\mathcal{V}P = \mathbb{C}\cdot P$. For each $i$, the matrix $PP_iP$ has
rank at most ${\rm rank}(P_i) \leq k$, so the only way it can be a scalar
multiple of $P$ is for it to be zero. But this implies that
$P = P(P_1 + \cdots + P_r)P = 0$, a contradiction.
\end{proof}

(If $r = {\rm dim}(\mathcal{V}) \leq k^2 - 1$ and each $P_i$ has rank
$k-1$, then this operator system has neither quantum $k$-cliques nor
quantum $k$-anticliques, explaining a parenthetical comment made
in the introduction.)

The minimum value of $r$ for which there exist $r$ projections, each of
rank at most $k$, which sum to $I_n$ is
$\lceil{\frac{n}{k}}\rceil$. Thus the following corollary is immediate.

\begin{coro}\label{ltupper}
$T^{\downarrow}(n,k) \leq \lceil{\frac{n}{k}}\rceil$.
\end{coro}

If $r = \lceil{\frac{n}{k}}\rceil$ then the operator system described
in Proposition \ref{projsum} is minimal in the sense that every operator
system properly contained in it does have a quantum $(k+1)$-anticlique. In
order to prove this, it will be useful to have the following alternative
characterization of quantum anticliques. (This characterization is implicit
in \cite{KLV}.)

\begin{lemma}\label{vectorlemma}
Let $\mathcal{V} \subseteq M_n(\mathbb{C})$ be an operator system. Then $\mathcal{V}$
has a quantum $k$-anticlique if and only if there exists an orthonormal
set $\{v_1, \ldots, v_k\}$ in $\mathbb{C}^n$ such that for every Hermitian
$A \in \mathcal{V}$
$$\langle Av_i, v_j\rangle = 0\qquad\mbox{and}\qquad
\langle Av_i,v_i\rangle = \langle Av_j,v_j\rangle$$
whenever $i \neq j$.
\end{lemma}

\begin{proof}
If there is a quantum $k$-anticlique $P$ for $\mathcal{V}$ then any orthonormal
basis $\{v_1, \ldots, v_k\}$ of its range is easily seen to have the stated
properties, since $PAP = \lambda P$ implies $\langle Av_i,v_j\rangle =
\langle PAPv_i,v_j\rangle = \lambda\langle v_i,v_j\rangle$ for all $i$
and $j$. Conversely, suppose we are given a set
$\{v_1, \ldots, v_k\}$ satisfying the conditions of the lemma and let $P$
be the orthogonal projection onto its span. Since every matrix in $\mathcal{V}$
is a linear combination of two Hermitian matrices in $\mathcal{V}$, the
stated equations will be true of any matrix in $\mathcal{V}$. So fix a matrix
$A \in \mathcal{V}$ and let $\lambda$ be the common value of the inner
products $\langle Av_i,v_i\rangle$. Then $P = \sum v_iv_i^*$ and so
$$PAP = \sum_{i,j} v_iv_i^* A v_j v_j^* =
\sum \lambda v_iv_i^* = \lambda P,$$
since $v_i^*Av_j = \langle Av_j,v_i\rangle$ is $0$ when $i \neq j$ and
$\lambda$ when $i = j$. Thus $PAP$ is a scalar multiple of $P$ for every
$A \in \mathcal{V}$, i.e., $P$ is a quantum anticlique.
\end{proof}

\begin{prop}\label{vmin}
Let $P_1, \ldots, P_r$ be orthogonal projections in $M_n(\mathbb{C})$ satisfying
$P_1 + \cdots + P_r = I_n$. Then any operator system properly contained in
$\mathcal{V} = {\rm span}(P_1, \ldots, P_r)$ has a quantum
$k$-anticlique where $k$ is the sum of the two smallest ranks of the
$P_i$'s.
\end{prop}

\begin{proof}
Let $\mathcal{V}_0$ be an operator system properly contained in $\mathcal{V}$.
Its Hermitian part $\mathcal{V}_0^h$ has the form
$$\mathcal{V}_0^h =
\left\{\sum a_iP_i: \vec{a} = (a_1, \ldots, a_r) \in E\right\}$$
where $E$ is some proper subspsace of $\mathbb{R}^r$ which includes the
vector $(1, \ldots, 1)$ (since we require $I_n \in \mathcal{V}_0$). So we
can find a nonzero $\vec{b} \in \mathbb{R}^r$ such that
$\vec{a}\cdot\vec{b} = 0$ for all $\vec{a} \in E$. Since $(1, \ldots, 1)
\in E$, it follows that $\vec{b}$ contains both strictly positive and
strictly negative components; by rearranging, we can assume that
$b_1, \ldots, b_j > 0$ and $b_{j+1}, \ldots, b_r \leq 0$. We can also assume
that $b_1 + \cdots + b_j = -b_{j+1} - \cdots - b_r = 1$.

For each $i$ let $e_{i,1}, \ldots, e_{i,{\rm rank}(P_i)}$ be an orthonormal
basis of ${\rm ran}(P_i)$. Let $k_1$ be the smallest rank among
$P_1, \ldots, P_j$ and let $k_2$ be the smallest rank among $P_{j+1}, \ldots,
P_r$, so that $k \leq k_1 + k_2$.
Then for $1 \leq l \leq k_1$ set
$v_l = \sqrt{b_1}e_{1,l} + \cdots + \sqrt{b_j}e_{j,l}$, and for
$1 \leq l \leq k_2$ set
$v_{k_1 + l} = \sqrt{-b_{j+1}}e_{j+1,l} + \cdots + \sqrt{-b_r}e_{r,l}$.
The vectors $v_l$ form an orthonormal set of size $k_1 + k_2$. For any
$A = a_1P_1 + \cdots + a_rP_r \in \mathcal{V}_0^h$ we then have
$\langle Av_l,v_{l'}\rangle = 0$ whenever $l \neq l'$, and for any
$1 \leq l \leq k_1$ and $k_1 + 1 \leq l' \leq k_1 + k_2$ we also have
$$\langle Av_l,v_l\rangle = a_1b_1 + \cdots + a_jb_j =
-a_{j+1}b_{j+1} - \cdots - a_rb_r = \langle Av_{l'},v_{l'}\rangle.$$
So Lemma \ref{vectorlemma} implies that $\mathcal{V}_0$ has a quantum
$(k_1 + k_2)$-anticlique.
\end{proof}

\begin{coro}\label{mincor}
Let $\mathcal{V}$ be the operator system from Proposition \ref{projsum}
and assume that $r = \lceil{\frac{n}{k}}\rceil$. Then every operator
system properly contained in $\mathcal{V}$ has a quantum $(k+1)$-anticlique.
\end{coro}

\begin{proof}
Any family of projections, each of rank at most $k$, which sums to $I_n$
must contain at least $r = \lceil{\frac{n}{k}}\rceil$ members. Thus if
it contains exactly this many members then the
sum of the two smallest ranks of the $P_i$'s must be at least $k+1$, as
otherwise these two projections could be replaced by a single projection
of rank at most $k$. The conclusion now follows from Proposition \ref{vmin}.
\end{proof}

It is not to be expected that the kind of minimality expressed in Corollary
\ref{mincor} can only happen at dimension $\lceil{\frac{n}{k}}\rceil$.
The following is an easy counterexample.

\begin{exam}\label{nonmin}
Take $n = 6$ and suppose $P_1 + P_2 = I_6$ where ${\rm rank}(P_1) =
{\rm rank}(P_2) = 3$. Then by Propositions \ref{projsum} and \ref{vmin},
${\rm span}(P_1,P_2)$ is a two-dimensional operator system with no quantum
$4$-anticliques, but every operator system properly contained in it
(there is only one, namely $\mathbb{C}\cdot I_6$) has a quantum $4$-anticlique.

Alternatively, suppose $Q_1 + Q_2 + Q_3 = I_6$ where ${\rm rank}(Q_1) =
{\rm rank}(Q_2) = {\rm rank}(Q_3) = 2$. Then ${\rm span}(Q_1,Q_2,Q_3)$
is a three-dimensional operator system which has no quantum $4$-anticliques
(Proposition \ref{projsum}), but I claim that any operator system properly
contained in it does have a quantum $4$-anticlique. To see this, note first
that
any two-dimensional operator system in $M_6(\mathbb{C})$ equals ${\rm span}(I_6,A)$ for
some Hermitian matrix $A$, and if it is contained in ${\rm span}(Q_1,Q_2,Q_3)$
then we can write $A = aQ_1 + bQ_2 + cQ_3$ for some $a,b,c \in \mathbb{R}$.
Without loss of generality assume $a \leq b \leq c$. If either $a = b$ or
$b = c$ then the existence of a quantum $4$-anticlique is immediate.
Otherwise let $\{e_{i,1}, e_{i,2}\}$ be an orthonormal basis for
${\rm ran}(Q_i)$ ($i = 1,2,3$) and set $\alpha = \frac{c-b}{c-a}$ and
$\gamma = \frac{b-a}{c-a}$, so that $\alpha + \gamma = 1$ and
$a\alpha + c\gamma = b$. Then the set $S = \{\sqrt{\alpha}e_{1,1} +
\sqrt{\gamma}e_{3,1}, \sqrt{\alpha}e_{1,2} + \sqrt{\gamma}e_{3,2},
e_{2,1}, e_{2,2}\}$ satisfies the conditions given in Lemma \ref{vectorlemma},
so ${\rm span}(I_6,A)$ has a quantum $4$-anticlique.
\end{exam}

In general, for any Hermitian $A \in M_n(\mathbb{C})$, a straightforward modification
of the argument used in this example shows that we can always find a quantum
$\lceil{\frac{n}{2}}\rceil$-anticlique for the two-dimensional
operator system $\mathcal{V} = {\rm span}(I_n,A)$. Let us record this fact:

\begin{prop}\label{singlematrix}
Let $A \in M_n(\mathbb{C})$ be Hermitian. Then ${\rm span}(I_n,A)$ has a quantum
$\lceil{\frac{n}{2}}\rceil$-anticlique.
\end{prop}

This is proven by ordering the eigenvalues of $A$ as $\lambda_1 \leq
\cdots \leq \lambda_n$, then letting $r = \lceil{\frac{n}{2}}\rceil$ and
for $1 \leq i \leq r - 1$ finding a convex combination
$\alpha_i\lambda_i + \alpha_{r + i}\lambda_{r + i} = \lambda_r$, and
then applying Lemma \ref{vectorlemma} to the vectors
$\sqrt{\alpha_i}v_i + \sqrt{\alpha_{r+i}}v_{r+1}$ plus the one additional
vector $v_r$, where $v_i$ is the eigenvector belonging to
$\lambda_i$.

In Example \ref{nonmin} this number is improved to
$\lceil{\frac{n}{2}}\rceil + 1$ because the two middle eigenvalues of $A$
are equal and their corresponding eigenvectors can both be used separately.

Actually, Tur\'{a}n's theorem does not just give the minimum number of
edges in a $(k+1)$-anticliqueless graph on $n$ vertices, it explicitly
describes the structure of such a graph with that minimum number of
edges --- and there is only one up to
isomorphism. I do not know whether $\lceil{\frac{n}{k}}\rceil$ is
the minimum dimension of a quantum $(k+1)$-anticliqueless operator system
in $M_n(\mathbb{C})$, but the operator system described in Proposition \ref{projsum}
with $r = \lceil{\frac{n}{k}}\rceil$ is not the only quantum
$(k+1)$-anticliqueless operator system of that dimension. We can see this
from the following extension of Proposition \ref{projsum}.

\begin{prop}\label{generic}
Let $A_1, \ldots, A_r$ be positive matrices in $M_n(\mathbb{C})$, each of
rank at most $k$, and suppose that the dimension of
${\rm ker}(\sum A_i)$ is also at most $k$. Then the operator system
$\mathcal{V} =
{\rm span}(I_n, A_1, \ldots, A_r)$ has no quantum $(k+1)$-anticliques.
\end{prop}

\begin{proof}
As in the proof of Proposition \ref{projsum}, if we assume that $P$ is a
quantum $(k+1)$-anticlique for $\mathcal{V}$ then comparing ranks shows that
$PA_iP = 0$ for all $i$. Thus $P(\sum A_i)P = 0$, which implies that
$(\sum A_i)^{1/2}P = 0$ and hence that $(\sum A_i)P = (\sum A_i)^{1/2}
(\sum A_i)^{1/2}P = 0$. This shows that ${\rm ran}(P)$ is contained in
${\rm ker}(\sum A_i)$, which contradicts the hypothesis
that ${\rm dim}({\rm ker}(\sum A_i)) \leq k$.
\end{proof}

Thus there are many operator systems of dimension $\lceil{\frac{n}{k}}\rceil$
which have no quantum $(k+1)$-anticliques. Indeed, if $A_1, \ldots, A_r$ are
positive matrices of rank $k$, where $r = \lceil{\frac{n}{k}}\rceil - 1$,
then {\it generically} the kernel of their sum will have dimension at
most $k$ and Proposition \ref{generic} will apply.

Now let us turn to lower bounds for $T^{\downarrow}(n,k)$. The next pair
of results are basically \cite[Theorems 3 and 4]{KLV}, with two small
improvements. For the reader's convenience I include the full proofs.

\begin{lemma}\label{klvlem}
Let $\mathcal{V}$ be an operator system in $M_n(\mathbb{C})$ and let
$d = {\rm dim}(\mathcal{V})$. Assume every matrix in $\mathcal{V}$ is
diagonal. If $(k-1)d + 1\leq n$ then $\mathcal{V}$ has a quantum
$k$-anticlique.
\end{lemma}

\begin{proof}
Write $\mathcal{V} = {\rm span}(A_1, \ldots, A_d)$ with each $A_i$ Hermitian
and $A_1 = I_n$. Then for each $1 \leq j \leq n$ let
$\vec{b}_j \in \mathbb{R}^{d-1}$ be the vector whose components are the
$(j,j)$ entries of $A_2$, $\ldots$, $A_d$. That is, $\vec{b}_j$ is the
sequence of eigenvalues of the $A_i$, excepting $A_1 = I_n$, belonging
to the $j$th standard basis vector $e_j$. By a theorem of Tverberg
\cite{T1, T2}, if $n \geq kd - (d-1) = (k-1)d + 1$, then the index set
$\{1, \ldots, n\}$ can be partitioned into $k$ blocks $S_1, \ldots, S_k$
such that the convex hulls of the sets $\{\vec{b}_j: j \in S_l\}
\subset \mathbb{R}^{d-1}$, for
$1 \leq l \leq k$, have nonempty intersection. That is, we can find a single
point $\vec{b} \in \mathbb{R}^{d-1}$ such that for each $1 \leq l \leq k$ some
convex combination $\sum_{j \in S_l} \mu_j\vec{b}_j$ equals $\vec{b}$.
Letting $v_l = \sum_{j \in S_l} \sqrt{\mu_j}e_j$, we then have that
$\langle A_iv_l, v_{l'}\rangle = 0$ whenever $l \neq l'$, for any $i$
(even $i = 1$), and if $i \neq 1$ then $\langle A_iv_l,v_l\rangle$ equals
the $i$th component of $\vec{b}$, while $\langle A_1v_l,v_l\rangle = 1$ for
any $l$. So $\mathcal{V}$ has a quantum $k$-anticlique by Lemma
\ref{vectorlemma}.
\end{proof}

\begin{theo}\label{klvthm}
Let $\mathcal{V}$ be an operator system in $M_n(\mathbb{C})$ and let
$d = {\rm dim}(\mathcal{V})$. If $(k-1)d + 1 \leq \lceil{\frac{n}{d-1}}\rceil$
then $\mathcal{V}$ has a quantum $k$-anticlique.
\end{theo}

\begin{proof}
We reduce to Lemma \ref{klvlem} by compressing $\mathcal{V}$ to an operator
system which contains only diagonal matrices. To do this, write $\mathcal{V} =
{\rm span}(A_1, \ldots, A_d)$ with each $A_i$ Hermitian and $A_1 = I_n$.
Start the construction by letting $v_1$ be a norm 1 eigenvector of $A_2$.
Then let $E_1 = \mathcal{V}v_1 = \{Bv_1: B \in \mathcal{V}\}$ and
let $P_1$ be the orthogonal projection onto $E_1^\perp$. Having constructed
$v_j$, $E_j$, and $P_j$, let $v_{j+1} \in {\rm ran}(P_j)$ be a norm 1
eigenvector for $P_jA_2P_j$, let $E_{j+1} = P_j\mathcal{V}v_{j+1}$, and let
$P_{j+1}$ be the orthogonal projection onto $(E_1 + \cdots + E_j)^\perp$.
Continue until all of $\mathbb{C}^n$ is exhausted.

Since $v_j$ is an eigenvector for $P_{j-1}A_2P_{j-1}$ (setting $P_0 = I_n$),
and also $P_{j-1}A_1P_{j-1}v_j = v_j$, it follows that the
dimension of
$E_j$ is at most $d-1$. Thus we have a sequence $(v_1, \ldots, v_r)$
with $r \geq \lceil{\frac{n}{d-1}}\rceil$. Also, by construction
$A_iv_j$ is orthogonal to $v_{j'}$ when $j < j'$, for any $i$.
Thus if $P$ is the orthogonal projection onto the span of the $v_j$'s,
then the matrices $PA_iP$ are diagonal with respect to the $v_j$ basis.
In other words, $P\mathcal{V}P$ satisfies the hypotheses of Lemma
\ref{klvlem} with $r \geq \lceil{\frac{n}{d-1}}\rceil$ in place of $n$.
So $(k-1)d + 1 \leq \lceil{\frac{n}{d-1}}\rceil$ implies that
$P\mathcal{V}P$ has a quantum $k$-anticlique, and hence that $\mathcal{V}$
does as well.
\end{proof}

The only novel aspects of these two proofs are (1) elimination of the
first coordinates of the vectors $\vec{b}_j$ in Lemma \ref{klvlem} and
(2) our choice of $v_j$ to be an eigenvector of $P_{j-1}A_2P_{j-1}$. Both
yield small improvements on the inequality that has to be assumed, meaning
that in both cases the inequality is slightly weakened.

Replacing $\lceil{\frac{n}{d-1}}\rceil$ with $\frac{n}{d-1}$ yields, if
anything, a stronger condition on $k$. So $(k-1)d + 1 \leq \frac{n}{d-1}$
implies that $\mathcal{V}$ has a quantum
$k$-anticlique. Substituting $k+1$ for $k$ and solving for $d$ yields the
condition $d \leq \frac{k-1 + \sqrt{(k+1)^2 + 4kn}}{2k}$ and thus any
operator system whose dimension is at most this value must have a quantum
$(k+1)$-anticlique. As
$$\sqrt{\frac{n}{k}} = \frac{\sqrt{4kn}}{2k} \leq
\frac{k-1 + \sqrt{(k+1)^2 + 4kn}}{2k},$$
$T^{\downarrow}(n,k)$ must be larger than this value.
Together with Corollary \ref{ltupper}, this yields the following estimate.

\begin{theo}\label{lowerest}
$\sqrt{\frac{n}{k}} < T^{\downarrow}(n,k) \leq \lceil{\frac{n}{k}}\rceil$.
\end{theo}

The more precise lower bound $\frac{k-1 + \sqrt{(k+1)^2 + 4kn}}{2k}$
is only marginally better than $\sqrt\frac{n}{k}$. But
when $k = 1$ it does improve $T^{\downarrow}(n,1) > \sqrt{n}$ to
$T^{\downarrow}(n,1) > \sqrt{n+1}$.

The obvious inefficiency in the proof of Theorem
\ref{klvthm}, where we start by compressing to a diagonal operator system,
plus the minimality demonstrated in Corollary \ref{mincor}, make it natural
to conjecture that the lower quantum Tur\'{a}n number $T^{\downarrow}(n,k)$
exactly equals $\lceil\frac{n}{k}\rceil$. When $n = 3$ and $k = 1$,
Proposition \ref{singlematrix} and Corollary \ref{ltupper} yield
$T^{\downarrow}(3,1) = 3$, so the first interesting case is
$n = 4$, $k = 1$, when Theorem \ref{lowerest} yields
$3 \leq T^{\downarrow}(4,1) \leq 4$ and the natural conjecture is
$T^{\downarrow}(4,1) = 4$, i.e., that every three-dimensional operator system
in $M_4(\mathbb{C})$ has a quantum $2$-anticlique.
But even this special case seems hard. I have only been able to prove
two partial positive results. The first is an immediate consequence of
either Corollary \ref{mincor} or Lemma \ref{klvlem}. (It can also be inferred
from Theorem \ref{m4thm} below.)

\begin{prop}
Let $\mathcal{V}$ be an operator system in $M_4(\mathbb{C})$ consisting of diagonal
matrices, and whose dimension is at most $3$. Then $\mathcal{V}$ has a
quantum $2$-anticlique.
\end{prop}

The other partial result is more substantive. Its content resides almost
entirely in the next lemma, which is a slightly modified version of a
theorem of Bryant \cite{B1}.

\begin{lemma}\label{rank2lemma}
Let $B_1, B_2, B_3 \in M_2(\mathbb{C})$ with $B_1$ and $B_3$ Hermitian and let $a \geq 1$.
Then there exist $\lambda \in [0,1]$ and $U \in SU(2)$ such that
$$SB_1S + CUB_2S + SB_2^*U^*C + CUB_3U^*C$$
is a scalar multiple of $I_2$, where $S = {\rm diag}(\sqrt{\lambda},
\sqrt{\lambda/a})$ and $C = {\rm diag}(\sqrt{1-\lambda},
\sqrt{1 - \lambda/a})$.
\end{lemma}

\begin{proof}
For any unit vector $\vec{z} = (z_0,z_1) \in \mathbb{C}^2$ we
have a special unitary matrix $U_{\vec{z}} =
\left[\begin{matrix}z_0&-\bar{z}_1\cr
z_1&\bar{z}_0\end{matrix}\right]$, and this identifies the 3-sphere
$S^3$ with $SU(2)$. Define $f_\lambda: SU(2) \to M_2(\mathbb{C})^h$ by
$f_\lambda(U) = SB_1S + CUB_2S + SB_2^*U^*C + CUB_3U^*C$, with $S$ and $C$ as
given above. (Recall that $M_2(\mathbb{C})^h$ is the Hermitian part of $M_2(\mathbb{C})$.)
Also define $g: M_2(\mathbb{C})^h \to \mathbb{R}\oplus \mathbb{C}$ by
$$g(A) = (a_{11} - a_{22}, 2a_{12})$$
where
$A = \left[\begin{matrix}a_{11}& a_{12}\cr a_{21}& a_{22}\end{matrix}\right]$.
Note that $g$ is real-linear and
$g(A) = 0$ if and only if $A$ is a scalar multiple of $I_2$.
Finally, let $F_\lambda: SU(2) \to \mathbb{R}\oplus \mathbb{C}$
be the map $F_\lambda = g\circ f_\lambda$.

If $B_3$ is a scalar multiple of $I_2$ then $f_0(U) = UB_3U^*$ is a scalar
multiple of $I_2$ for any $U \in SU(2)$, and we are done. So assume this is
not the case.
Since $S^2 + C^2 = I_2$, adding a scalar multiple of $I_2$ to both
$B_1$ and $B_3$ does not change the problem, so we can assume one of the
eigenvalues of $B_3$ is $0$. Multiplying $B_1$, $B_2$, and $B_3$ by a
nonzero scalar,
we can assume the other eigenvalue is $1$. We can then find $V \in SU(2)$
such that $VB_3V^* = {\rm diag}(1,0)$, and if $U$ solves the problem for
$B_1$, $VB_2$, and $VB_3V^*$ then $UV$ solves the problem for $B_1$, $B_2$,
and $B_3$. So we may assume $B_3 = {\rm diag}(1,0)$.

To reach a contradiction, suppose $F_\lambda(U) \neq 0$ for all
$\lambda \in [0,1]$ and $U \in SU(2)$. Then we can define
$\tilde{F}_\lambda: SU(2) \to S^2 \subset \mathbb{R}\oplus \mathbb{C}
\cong \mathbb{R}^3$ by
$\tilde{F}_\lambda(U) = \frac{F_\lambda(U)}{|F_\lambda(U)|}$.

The family of maps $\tilde{F}_\lambda$ constitutes a homotopy from
$\tilde{F}_0$ to $\tilde{F}_1$. Now
$S = 0$ and $C = I_2$ when $\lambda = 0$, so that $f_0(U) = UB_3U^*$.
Recalling that we have reduced to the case where $B_3 = {\rm diag}(1,0)$,
a short computation shows that $\tilde{F}_0(U_{\vec{z}}) = F_0(U_{\vec{z}}) =
(|z_0|^2 - |z_1|^2, 2z_0\bar{z}_1)$, i.e., it is the Hopf map from $S^3$
to $S^2$.

This map is homotopically nontrivial, so
to generate a contradiction we need only to show that $\tilde{F}_1$ is
null homotopic. When $\lambda = 1$ we have $C = {\rm diag}(0,a')$
with $a' = \sqrt{1 - 1/a}$. So $F_1(U) \in \mathbb{R} \oplus
\mathbb{C}$ is a constant (namely $g(SB_1S)$) plus something real-linear in
the entries of $U$ (namely $g(CUB_2S + SB_2^*U^*C)$) plus
something in $\mathbb{R} \oplus 0$ (namely $g(CUB_3U^*C)$). Letting
$X = \{U \in SU(2): F_1(U)\in \mathbb{R} \oplus 0\}$, it follows that $X$ is
the intersection of $SU(2) \cong S^3$ with an affine real-linear subspace
of $\left\{\left[\begin{matrix}\alpha&-\bar{\beta}\cr\beta&\bar{\alpha}\end{matrix}\right]: \alpha,\beta \in \mathbb{C}\right\} \cong \mathbb{R}^4$
whose real dimension is at least $2$. Thus $X$ is connected, and therefore its
image under $F_1$
in $\mathbb{R} \oplus 0$ is connected. Since this image does not
contain $0$, it must therefore lie entirely in $(0,\infty) \oplus 0$ or
$(-\infty,0) \oplus 0$; in either case, the image of $\tilde{F}_1$ cannot
be all of $S^2$ and so $\tilde{F}_1$ must be null homotopic. This contradicts
the homotopic nontriviality of $\tilde{F}_0$, and we conclude that
$g(f_\lambda(U)) = F_\lambda(U) \in \mathbb{R} \oplus \mathbb{C}$ must be
$0$ for some $\lambda \in [0,1]$ and $U \in SU(2)$. So $f_\lambda(U)$ is
a scalar multiple of $I_2$ for this $\lambda$ and $U$.
\end{proof}

\begin{theo}\label{m4thm}
Let $A,B \in M_4(\mathbb{C})$ be Hermitian and assume $A$ has a repeated
eigenvalue.
Then $\mathcal{V} = {\rm span}(I_4,A,B)$ has a quantum $2$-anticlique.
\end{theo}

\begin{proof}
If $A$ has a triple eigenvalue then there is a rank $3$ orthogonal projection
$P$ such that $PAP$ is a scalar multiple of $P$. We can then identify
$PM_4(\mathbb{C})P$ with $M_3(\mathbb{C})$ and invoke
Proposition \ref{singlematrix} to infer that ${\rm span}(I_3,PBP)$ has a
quantum $2$-anticlique $Q$. This $Q$ will then be a quantum $2$-anticlique
for $\mathcal{V}$.

So assume $A$ has an eigenvalue of multiplicity exactly $2$. By adding a
scalar multiple of $I_4$ to $A$, we can assume that this eigenvalue is $0$.
There are now two cases to consider. First, suppose the two nonzero eigenvalues
of $A$ have opposite sign. Without loss of generality say
$A = {\rm diag}(a,-b,0,0)$ with $a,b > 0$. Multiplying $A$ by a nonzero
scalar, we can also assume that $\frac{1}{a} + \frac{1}{b} = 1$. Then let
$W = \left[\begin{matrix}\frac{1}{\sqrt{a}}&0&0\cr
\frac{1}{\sqrt{b}}&0&0\cr
0&1&0\cr
0&0&1
\end{matrix}\right]$, so that $W^*W = I_3$ and $W^*AW = 0$. Again by
Proposition \ref{singlematrix},
${\rm span}(I_3, W^*BW) \subset M_3(\mathbb{C})$ has a quantum $2$-anticlique $Q$,
and $P = WQW^*$ is then a quantum $2$-anticlique for $\mathcal{V}$.

In the other case, the two nonzero eigenvalues of $A$ have the same sign.
Multiplying by a scalar and diagonalizing, we can assume that
$A = {\rm diag}(1,a,0,0)$ with $a \geq 1$. In this basis write
$B = \left[\begin{matrix}B_1&B_2^*\cr
B_2&B_3
\end{matrix}\right]$ with $B_1, B_2,B_3 \in M_2(\mathbb{C})$ and $B_1$ and $B_3$
Hermitian. Then find $\lambda$ and $U$ as in Lemma \ref{rank2lemma} and
define
$$P = \left[\begin{matrix}S^2&SCU\cr
U^*SC&U^*C^2U
\end{matrix}\right],$$
with $S$ and $C$ as in the statement of that lemma. A computation
now shows that both $PAP$ and $PBP$ are scalar multiples of $P$. To see
that ${\rm rank}(P) = 2$, observe that $P$ is unitarily conjugate to
$\left[\begin{matrix}S^2&SC\cr SC&C^2\end{matrix}\right]$, which after
interchanging the middle two basis vectors is the
direct sum of $\left[\begin{matrix}\sqrt{\lambda}\cr\sqrt{1-\lambda}\end{matrix}\right]\left[\begin{matrix}\sqrt{\lambda}&\sqrt{1-\lambda}\end{matrix}\right]$
and $\left[\begin{matrix}\sqrt{\lambda/a}\cr\sqrt{1-\lambda/a}\end{matrix}\right]\left[\begin{matrix}\sqrt{\lambda/a}&\sqrt{1-\lambda/a}\end{matrix}\right]$.
\end{proof}

In other words, any three-dimensional operator system in $M_4(\mathbb{C})$ has a
quantum $2$-anticlique provided it contains a nonscalar matrix that has
a repeated eigenvalue. Unfortunately, for generic Hermitian $A,B \in M_4(\mathbb{C})$
the operator system ${\rm span}(I_4,A,B)$ does not have this property
\cite{B2}.

\section{Upper quantum Tur\'{a}n numbers}

We define the {\it upper quantum Tur\'{a}n number} $T^{\uparrow}(n,k)$
to be the largest number $d$ such that some operator system in $M_n(\mathbb{C})$
whose dimension is $d$ has no quantum $(k+1)$-cliques.
As before, we restrict attention to the case $k \geq 1$.

Evaluating $T^{\uparrow}(n,k)$ and $T^{\downarrow}(n,k)$ are very different
problems. In general there is no natural ``quantum'' analog of edge
complementation which would interchange quantum cliques and anticliques.
In finite dimensions we can consider the orthocomplement $\mathcal{V}^{\perp}$
of an operator system $\mathcal{V} \subseteq M_n(\mathbb{C})$ relative to the
Hilbert-Schmidt inner product $\langle A,B\rangle = {\rm tr}(AB^*)$, but
it will not contain $I_n$. In order to produce a
``complementary'' operator system we could
define $\mathcal{V}^\dag = \mathcal{V}^{\perp} + \mathbb{C}\cdot I_n$,
and this is a genuine complementation operation in the sense that
$\mathcal{V}^{\dag\dag} = \mathcal{V}$. This operation transforms quantum
anticliques into quantum cliques, but not vice versa (incidentally
making precise the idea that anticliques are more special than cliques).
We can infer from this fact that
$T^{\uparrow}(n,k) \leq n^2 + 1 - T^{\downarrow}(n,k)$, but this upper
bound is terrible compared to the one proven below.\footnote{Maybe {\it quantum clique for $\mathcal{V}$} should be redefined to simpy mean {\it quantum anticlique for $\mathcal{V}^\dag$}? This would automatically introduce a symmetry between quantum cliques and quantum anticliques, but it suffers from two drawbacks: first, it does not generalize to the infinite-dimensional setting, and second, the quantum Ramsey theorem from \cite{W3} would fail. According to \cite[Proposition 2.1]{W3} the diagonal operator system $D_n$ has no quantum $2$-anticliques, but $D_n^\dag$ also has no quantum $2$-anticliques. (Suppose $P$ is a quantum $2$-anticlique for $D_n^\dag$. Let $v_1$ and $v_2$ be orthonormal vectors in ${\rm ran}(P)$ and consider the operator $A: v \mapsto \langle v,v_1\rangle v_2$. Then $A = PAP$ and ${\rm tr}(A) = 0$, so that ${\rm tr}(AB^*) = {\rm tr}(APB^*P) = 0$ for all $B \in D_n^\dag$, which implies that $A \in D_n$. But $A$ cannot belong to $D_n$ because it does not commute with $A^*$, contradiction.)}

For $k = 1$, evaluation of $T^{\uparrow}(n,1)$ is not trivial, but it is
completely solved:

\begin{theo}\label{citedthm}
(\cite[Theorem 3.3]{W3}) For any $n \geq 2$, $T^{\uparrow}(n,1) = 3$.
\end{theo}

(The cited result only states that $T^{\uparrow}(n,1) < 4$, but the
reverse inequality follows from the trivial lower bound
$T^{\uparrow}(n,k) \geq (k+1)^2 - 1$. If ${\rm dim}(\mathcal{V}) < (k+1)^2$
then $\mathcal{V}$ obviously cannot have any quantum $(k+1)$-cliques.)

In contrast, it follows from \cite[Proposition 2.3]{W3} that
$T^{\uparrow}(n,2) \to \infty$ as $n \to \infty$. The example which shows
this can be
described more abstractly, in a way that generalizes to larger values of $k$.

\begin{prop}\label{upperlower}
Let $Q$ be an orthogonal projection in $M_n(\mathbb{C})$ of rank $n - k + 1$. Then the
operator system
$$\mathcal{V}_Q = \{A \in M_n(\mathbb{C}): QAQ\mbox{ is a scalar multiple of }Q\}$$
has no quantum $(k+1)$-cliques. Indeed, no two-dimensional extension of
$\mathcal{V}_Q$ has any quantum $(k+1)$-cliques, but every three-dimensional
extension of $\mathcal{V}_Q$ does have a quantum $(k+1)$-clique.
\end{prop}

\begin{proof}
Let $P$ be a rank $k+1$ orthogonal projection. Then since
${\rm rank}(P) + {\rm rank}(Q) = n + 2$ there is a rank 2 orthogonal
projection $P_0$ which lies below both $P$ and $Q$. Since $Q$ is a quantum
anticlique for $\mathcal{V}_Q$, so is $P_0$, i.e.,
${\rm dim}(P_0\mathcal{V}_QP_0) = 1$. Thus any two-dimensional extension
$\mathcal{V}_Q'$ of $\mathcal{V}_Q$ must satisfy
${\rm dim}(P_0\mathcal{V}_Q'P_0) \leq 3$, so that $P_0$ cannot be a
quantum $2$-clique for $\mathcal{V}_Q'$. This implies that $P$ cannot be
a quantum $(k+1)$-clique for $\mathcal{V}_Q'$.

Now let $\mathcal{V}_Q''$ be a three-dimensional extension of $\mathcal{V}_Q$.
Then ${\rm dim}(Q\mathcal{V}_Q''Q) = 4$. (Consider the map $F: A \mapsto QAQ$
from $M_n(\mathbb{C})$ to $QM_n(\mathbb{C})Q$. We have $\mathcal{V}_Q = {\rm ker}(F) + \mathbb{C}\cdot I_n$, so if $\mathcal{V}$ is a $d$-dimensional extension of $\mathcal{V}_Q$
then ${\rm dim}(F(\mathcal{V})) = d + 1$.) So by Theorem \ref{citedthm}
$Q\mathcal{V}_Q''Q$ has a quantum $2$-clique $Q_0$, and I claim that
the projection $P = (I-Q) + Q_0$ is then a quantum $(k+1)$-clique for
$\mathcal{V}_Q''$. To see this, let $A \in M_n(\mathbb{C})$ be any matrix
which satisfies $PAP = A$; we must show that $A \in P\mathcal{V}_Q''P$.
Since $Q_0$ is a quantum $2$-clique for $\mathcal{V}_Q''$, we can find
$B_0 \in \mathcal{V}_Q''$ such that $Q_0B_0Q_0 = Q_0AQ_0$. Let $B_1 =
QB_0Q$; then $Q(B_1 - B_0)Q = 0$ and so $B_1 - B_0 \in \mathcal{V}_Q$,
which implies that $B_1 \in \mathcal{V}_Q''$. Similarly,
$Q(A - Q_0AQ_0)Q = Q(PAP - Q_0AQ_0)Q = 0$ so $A - Q_0AQ_0 \in \mathcal{V}_Q$,
and finally
$B = A - Q_0AQ_0 + B_1$ belongs to $\mathcal{V}_Q''$ and satisfies $PBP = A$.
Thus we have shown that $P\mathcal{V}_Q''P$ contains $A$, as desired.
\end{proof}

The last part of Proposition \ref{upperlower} shows that the operator
systems $\mathcal{V}_Q'$ are maximal for not having any quantum
$(k+1)$-cliques. Of course, this does not rule out the possibility that
other operator systems whose dimensions are larger could lack quantum
$(k+1)$-cliques.

If $Q$ is diagonalized as $Q = {\rm diag}(0, \ldots, 0, 1, \ldots, 1)$
(with $k-1$ zeros and $n-k+1$ ones) then $\mathcal{V}_Q$ appears as the
set of matrices whose restriction to the bottom right $(n-k+1)\times(n-k+1)$
corner is a scalar multiple of the $(n-k+1)\times(n-k+1)$ identity matrix, and
which can be anything on the top and left $(k-1)\times n$ and $n\times (k-1)$
strips. Thus ${\rm dim}(\mathcal{V}_Q) = 2(k-1)n - (k-1)^2 + 1$ and we infer
the following corollary.

\begin{coro}\label{thirdbound}
$T^{\uparrow}(n,k) \geq 2(k-1)n - (k-1)^2 + 3$.
\end{coro}

The classical analog of the operator system $\mathcal{V}_Q$ is the graph
on $n$ vertices which is the edge complement of a single $(n-k+1)$-clique.
In other words, the only missing edges are those both of whose endpoints
lie within a fixed set of $n-k+1$ vertices. Such a graph contains no
$(k+1)$-cliques, but the number of edges is has is linear in $n$,
whereas the classical Tur\'{a}n numbers go like $n^2$.

We could try to get a better lower bound by considering the matrix analog
of a $(k+1)$-cliqueless graph with the maximal number of edges.
This graph is the edge complement of a disjoint union of $k$ many cliques of
equal or nearly equal size. The matrix analog would be the operator system
$\{A \in M_n(\mathbb{C}): P_iAP_i$ is a scalar multiple of $P_i$ for $1 \leq i \leq k\}$
where $P_1, \ldots, P_k$ are orthogonal projections of equal or nearly equal rank
which sum to $I_n$. But this idea does not work because this operator
system typically does have quantum $(k+1)$-cliques. This is most simply illustrated
in the case $k=2$ when we are dealing with a ``complete bipartite'' operator
system which might be expected to have no quantum $3$-cliques. This expectation
fails badly, however:

\begin{prop}
Let $\mathcal{V}_0 \subset M_{2k}(\mathbb{C})$ be the set of matrices of the form
$\left[\begin{matrix}0&A\cr B&0\end{matrix}\right]$ with $A,B \in M_k(\mathbb{C})$,
and let $\mathcal{V}$ be the operator system $\mathcal{V} = \mathcal{V}_0
+ \mathbb{C}\cdot I_{2k}$. Then $\mathcal{V}$ has a quantum $k$-clique.
\end{prop}

\begin{proof}
Let $E = \{v \oplus v: v \in \mathbb{C}^k\} \subset \mathbb{C}^{2k}$ and let
$P$ be the orthogonal projection onto $E$. Any linear operator from $E$ to
itself has the form $(v\oplus v) \mapsto (Av\oplus Av)$ for some $A \in M_k(\mathbb{C})$.
But for any $A \in M_k(\mathbb{C})$ the
matrix $A' = \left[\begin{matrix}0&A\cr A&0\end{matrix}\right]$ satisfies
$$(PA'P)(v \oplus v) = A'(v\oplus v) = Av \oplus Av,$$
so that $P\mathcal{V}P$ contains every linear operator from $E$ to itself.
That is, $P$ is a quantum $k$-clique.
\end{proof}

In fact, linearity in $n$ is the most we can ask for in a lower bound
on $T^{\uparrow}(n,k)$, because --- incredibly --- we can give an upper
bound on $T^{\uparrow}(n,k)$ which is also linear in $n$. The argument
uses the following result from \cite{W3}. Let $(e_i)$ be the standard
basis of $\mathbb{C}^n$.

\begin{lemma}\label{findclique}
(\cite[Lemma 4.4]{W3})
Let $n = k^4 + k^3 + k - 1$ and let $\mathcal{V}$ be an operator system
contained in $M_n(\mathbb{C})$. Suppose $\mathcal{V}$ contains matrices $A_1, \ldots,
A_{k^4 + k^3}$ such that for each $i$ we have $\langle A_ie_i, e_{i+1}\rangle
\neq 0$, and also $\langle A_i e_r,e_s\rangle = 0$ whenever ${\rm max}(r,s)
> i + 1$ and $r \neq s$. Then $\mathcal{V}$ has a quantum $k$-clique.
\end{lemma}

We need this lemma to prove the next result, which is extracted from the
proof of \cite[Theorem 4.5]{W3}. For the reader's convenience I include
the proof here.

\begin{lemma}\label{rangelemma}
Let $\mathcal{V}$ be an operator system in $M_n(\mathbb{C})$ and suppose that for
each nonzero $v \in \mathbb{C}^n$ we have ${\rm dim}(\mathcal{V}v) \geq
8k^8$. Then $\mathcal{V}$ has a quantum $k$-clique.
\end{lemma}

\begin{proof}
Let $v_1$ be any nonzero vector in $\mathbb{C}^n$ and find $A_1 \in
\mathcal{V}$ such that $v_2 = A_1v_1$ is nonzero and orthogonal to $v_1$.
Then find $A_2 \in \mathcal{V}$ such that $v_3 = A_2v_2$ is nonzero and
orthogonal to each of $v_1$, $A_1v_1$, $A_1^*v_1$, $A_1v_2$, and $A_1^*v_2$.
Continue in this way, at the $r$th step finding $A_r \in \mathcal{V}$ such
that $v_{r+1} = A_rv_r$ is nonzero and orthogonal to the span of the vectors
$v_1$ and $A_iv_j$ and $A_i^*v_j$ for $i < r$ and $j \leq r$. The dimension
of this span is at most $2r^2 - 2r + 1$, so as long as $r \leq 2k^4$ its
dimension is less than $8k^8$ and a suitable matrix $A_r$ can be found.
Compressing to the span of the $v_i$ for $1 \leq i \leq k^4 + k^3 + k -1$
then puts us in the situation of Lemma \ref{findclique}, so there is a
quantum $k$-clique by that result.
\end{proof}

\begin{theo}
Let $\mathcal{V}$ be an operator system in $M_n(\mathbb{C})$ of dimension
at least $16k^8n$. Then $\mathcal{V}$ has a quantum $k$-clique.
\end{theo}

\begin{proof}
Fix $k$; the proof goes by induction on $n$. The smallest sensible value of $n$
is $n = 16k^8$, as for smaller values of $n$ the dimension of $\mathcal{V}$ is
at most $n^2 < 16k^8n$. When $n$ exactly equals $16k^8$, the only way to have
${\rm dim}(\mathcal{V}) \geq 16k^8n$ is if $\mathcal{V} = M_n(\mathbb{C})$,
so it certainly has a quantum $k$-clique. In the induction step, first
suppose that there exists a nonzero vector $v \in \mathbb{C}^n$ such
that ${\rm dim}(\mathcal{V}v) < 8k^8$. Let $P$ be the rank $n-1$ orthogonal
projection onto the orthocomplement of $\mathbb{C}\cdot v$ in $\mathbb{C}^n$.
If $A \in \mathcal{V}$ satisfies $PAP = 0$ then, with respect to an orthonormal
basis of which $v$ is the first element, $A$ is the sum of a matrix which is
zero except on the leftmost column and a matrix which is zero except on the
topmost row. Since ${\rm dim}(\mathcal{V}v) < 8k^8$, it follows that
the set $\{A \in \mathcal{V}: PAP = 0\}$ has dimension at most
$16k^8$. Thus
$${\rm dim}(P\mathcal{V}P) \geq {\rm dim}(\mathcal{V}) - 16k^8
\geq 16k^8(n-1),$$
and the induction hypothesis tells us that $P\mathcal{V}P$ has a quantum
$k$-clique, so $\mathcal{V}$ does as well.

Otherwise, for every nonzero vector $v \in \mathbb{C}^n$ we have
${\rm dim}(\mathcal{V}v) \geq 8k^8$, and then $\mathcal{V}$ has a quantum
$k$-clique by Lemma \ref{rangelemma}.
\end{proof}

Putting this together with Corollary \ref{thirdbound} yields the promised
bounds on $T^{\uparrow}(n,k)$.

\begin{coro}
$2(k-1)n - (k-1)^2 + 3 \leq T^{\uparrow}(n,k) < 16(k+1)^8n$.
\end{coro}



\begin{thebibliography}{aaaaaaaa}

\bibitem{B1}
R.\ Bryant,
https://mathoverflow.net/questions/289375/2-times-2-matrix-question.

\bibitem{B2}
{---------},
https://mathoverflow.net/questions/289683/existence-of-double-eigenvalue.

\bibitem{DSW}
R.\ Duan, S.\ Severini, and A.\ Winter, Zero-error communication via
quantum channels, noncommutative graphs, and a quantum Lov\'{a}sz number,
{\it IEEE Trans.\ Inform.\ Theory \bf 59} (2013), 1164-1174. 

\bibitem{KKS}
M.\ Kennedy, T.\ Kolomatski, and D.\ Spivak, An infinite quantum Ramsey
theorem, arXiv:1711.09526.

\bibitem{KLV}
E.\ Knill, R.\ Laflamme, and L.\ Viola, Theory of quantum error
correction for general noise, {\it Phys.\ Rev.\ Lett.\ \bf 84} (2000),
2525-2528.

\bibitem{S}
D.\ Stahlke, Quantum source-channel coding and non-commutative graph
theory, arXiv:1405.5254.

\bibitem{T1}
H.\ Tverberg, A generalization of Radon's theorem, {\it J.\ London
Math.\ Soc.\ \bf 41} (1966), 123-128.

\bibitem{T2}
{---------}, A generalization of Radon's theorem, II, {\it Bull.\
Austral.\ Math.\ Soc.\ \bf 24} (1981), 321-325.

\bibitem{W}
N.\ Weaver, Quantum relations, {\it Mem.\ Amer.\ Math.\ Soc.\ \bf 215}
(2012), v-vi, 81-140.

\bibitem{W2}
{---------}, Quantum graphs as quantum relations, to appear in
{\it Proc.\ Roy.\ Soc.\ Edinburgh Sect.\ A}.

\bibitem{W3}
{---------}, A ``quantum'' Ramsey theorem for operator systems, to appear
in {\it Proc.\ Amer.\ Math.\ Soc.}


\end{thebibliography}
\end{document}